\documentclass[12pt]{amsart}

\usepackage{amssymb}
\usepackage[shortlabels]{enumitem}
\setlist[1]{wide}
\setlist[2]{leftmargin=15mm}
\setlist[enumerate]{label=\rm({\roman*}), }
\setlist[enumerate,2]{label=\rm{(\arabic*)}}
\setlist[itemize]{label=\raisebox{0.25ex}{\tiny$\bullet$}}

\usepackage[backref, colorlinks, linktocpage, citecolor = blue, linkcolor = blue]{hyperref} 
\usepackage{soul}
\usepackage[all]{xy}

\usepackage{xcolor} 
\theoremstyle{plain}

\newtheorem{theorem}{Theorem}[section]
\newtheorem{lemma}[theorem]{Lemma}
\newtheorem{proposition}[theorem]{Proposition}

\theoremstyle{definition}

\newtheorem{maintheorem}{Theorem}

\renewcommand\k{\mathbf{ \textbf k}}
\newcommand{\p}{\mathbb{P}}

\newcommand\bP{{\mathbb P}}

\newcommand\bZ{{\mathbb Z}}

\newcommand\cI{{\mathcal I}}

\newcommand\cN{{\mathcal N}}
\newcommand\cO{{\mathcal O}}

\newcommand\cT{{\mathcal T}}

\newcommand\gp{{\rm gp}}

\newcommand\id{{\rm id}}

\newcommand\pr{{\rm pr}}

\DeclareMathOperator\Aut{Aut}

\newcommand\Der{{\rm Der}}
\newcommand\End{{\rm End}}

\newcommand\Ker{{\rm Ker}}
\newcommand\Lie{{\rm Lie}}

\numberwithin{equation}{section}

\title[Abelian varieties as automorphism groups]
{Abelian varieties as automorphism groups of smooth projective
varieties in arbitrary characteristics}

\author{J\'er\'emy Blanc}
\address{J\'er\'emy Blanc, Universit\"at Basel, 
Departement Mathematik und Informatik, Spiegelgasse 1, 
CH-4051 Basel, Switzerland. }
\author{Michel Brion}
\address{Michel Brion, Universit\'e Grenoble Alpes, 
Institut Fourier, CS 40700, 38058 Grenoble Cedex 9, France}

\date{}

\thanks{The first author is supported by the Swiss National 
Science Foundation Grant ``Birational transformations of threefolds'' 
$200020\_178807$.  Both authors thank the referee for a careful
reading and helpful comments.}

\begin{document}

\begin{abstract}
Let $A$ be an abelian variety over an algebraically closed field. We show that
$A$ is the automorphism group scheme of some smooth projective variety if and
only if $A$ has only finitely many automorphisms as an algebraic group.
This generalizes a result of Lombardo and Maffei for complex abelian varieties.
\end{abstract}

\maketitle

\tableofcontents
\section{Introduction}
\label{sec:int}

Let $X$ be a projective algebraic variety over an algebraically closed field. 
The automorphism group functor of $X$ is represented by a group scheme 
$\Aut_X$, locally of finite type (see \cite[p.~268]{Grothendieck}
or \cite[Thm.~3.7]{MO}).
Thus, the automorphism group $\Aut(X)$ is the group of $k$-rational 
points of a smooth group scheme that we will still denote by $\Aut(X)$
for simplicity. One may ask which smooth group schemes are obtained 
in this way, possibly imposing some additional conditions on $X$ such 
as smoothness or normality. It is known that every finite group $G$ 
is the automorphism group scheme of some smooth projective curve $X$ 
(see e.g.~the main result of \cite{MR}).
The case of a complex abelian variety $A$ was treated recently by 
Lombardo and Maffei in \cite{LM}; they showed that 
$A = \Aut(X)$ for some complex projective manifold $X$ if and only if 
$A$ has only finitely many automorphisms as an algebraic group.
In this note, we generalize their result as follows:

\begin{maintheorem}\label{thm:main}
Let $A$ be an abelian variety over an algebraically closed field.
Denote by $\Aut_{\gp}(A)$ the group of automorphisms of $A$ as an algebraic group. 

\begin{enumerate}

\item\label{main1}
If $A = \Aut(X)$ for some projective variety $X$, then $\Aut_{\gp}(A)$
is finite.

\item\label{main2}
If $\Aut_{\gp}(A)$ is finite, then there exists a smooth projective
variety $X$ such that $A = \Aut_X$.

\end{enumerate}

\end{maintheorem}

Like in \cite{LM}, the proof of the first assertion is easy, and the 
second one is obtained by constructing $X$ as a quotient 
$(A \times Y)/G$, where $G \subset A$ is a finite subgroup, 
$Y$ is a smooth projective variety such that $G = \Aut_Y$, and 
the quotient is taken for the diagonal action of $G$ on $A \times Y$.
In \cite{LM}, $G$ is a cyclic group of prime order $\ell$, and $Y$
a surface of degree $\ell$ in $\bP^3$ equipped with a free action
of $G$. As the construction of $Y$ does not extend readily to 
prime characteristics, we take for $G$ the $n$-torsion subgroup 
scheme $A[n]$ for an appropriate integer $n$, and for $Y$ 
an appropriate rational variety. 

A different construction of a variety $X$ satisfying the second
assertion has been obtained independently by Mathieu Florence,
see \cite{Florence}; it works over an arbitrary field.
 
Let us briefly describe the structure of this note.  
Section~\ref{sec:prelim} is a short introduction to basic notation 
and reminders on abelian varieties. In Section~\ref{sec:(i)}, 
we take an abelian variety $A$ with $\Aut_{\gp}(A)$ infinite, 
assume that $A = \Aut(X)$ for some projective variety $X$, 
and derive a contradiction. In Section~\ref{sec:(ii)}, we take 
an abelian variety $A$ with $\Aut_{\gp}(A)$ finite and prove that 
for each prime number $\ell$ different from the characteristic 
of the ground field, for each $m\ge 1$ big enough, and for each
smooth rational projective variety $Y$ with $\Aut_Y\simeq A[\ell^m]$, 
one has
\[\Aut_X = A\]
where  $X$ is the smooth projective variety $(A \times Y)/A[\ell^m]$. 
Then, Section~\ref{Sec:Y} is devoted to an explicit construction of $Y$.

\section{Preliminaries and notation}\label{sec:prelim}

We begin by fixing some notation and conventions which will be 
used throughout this note. The ground field $\k$ is algebraically 
closed, of characteristic $p \geq 0$. A variety $X$ is a separated
integral scheme of finite type over~$k$. By a point of $X$,
we mean a $\k$-rational point.

We use \cite{Mumford} as a general reference for abelian varieties.
We denote by~$A$ such a variety of dimension $g \geq 1$, 
with group law $+$ and neutral element $0$. Then
\[ \Aut(A) = A \rtimes \Aut_{\gp}(A), \]
where $A$ acts on itself by translations.  
Moreover, $\Aut_{\gp}(A) = \Aut(A,0)$ (the group of automorphisms 
fixing the neutral element), see \cite[\S 4, Cor.~1]{Mumford}.

For any positive integer $n$, we denote by $A[n]$ the $n$-torsion
subgroup scheme of $A$, i.e., the schematic kernel of 
the multiplication map 
\[ n_A \colon A \longrightarrow A, \quad a \longmapsto n a. \]
Clearly, $A[n]$ is stable by $\Aut_{\gp}(A)$. Also, recall
from \cite[\S 6, Prop.]{Mumford} that $A[n]$ is finite; moreover,
$A[n]$ is the constant group scheme $(\bZ/n)^{2g}$ if $n$ is prime
to $p$.

We denote by 
\[ q \colon A \longrightarrow  A/A[n], \quad a \longmapsto \bar{a} \]
the quotient morphism. Then $n_A$ factors as $q$ followed by an isomorphism 
$A/A[n] \stackrel{\simeq}{\longrightarrow} A$.

\section{Proof of Theorem \ref{thm:main}\ref{main1}}
\label{sec:(i)}

In this section, we choose an abelian variety $A$ such
that $\Aut_{\gp}(A)$ is infinite, and proceed to
the proof of Theorem~\ref{thm:main}\ref{main1}. We will need:

\begin{lemma}\label{lem:ker}
For any positive integer $n$, the kernel of the restriction map
\[ \rho_n \colon \Aut_{\gp}(A) \longrightarrow \Aut_{\gp}(A[n]) \]
is infinite.
\end{lemma}

\begin{proof}
Note that $\rho_n$ extends to a ring homomorphism
\[ \sigma_n \colon \End_{\gp}(A) \longrightarrow \End_{\gp}(A[n]) \]
with an obvious notation. Moreover, the image of $\sigma_n$
is a finitely generated abelian group (as a quotient of 
$\End_{\gp}(A)$) and is killed by $n$; thus, this image
is finite. So the image of $\rho_n$ is finite as well. 
\end{proof}

We assume, for contradiction, the existence of a projective variety $X$ 
such that $A = \Aut(X)$; in particular, $X$ is equipped with a faithful 
action of $A$. By \cite[Lem.~3.2]{Brion}, there exist a finite subgroup
scheme $G$ of $A$ and an $A$-equivariant morphism
$f \colon X \to A/G$, where $A$ acts on $A/G$ via the quotient map.
Denote by $n$ the order of $G$; then $G$ is a subgroup scheme 
of $A[n]$. By composing $f$ with the natural map $A/G \to A/A[n]$,
we may thus assume that $G = A[n]$. 

We now adapt the proof of \cite[Thm.~2.2]{LM}.
Let $Y$ be the schematic fiber of $f$ at $\bar{0}$. Then $Y$
is a closed subscheme of $X$, stable by the action of $A[n]$.
Form the cartesian square 
\[ \xymatrix{
X' \ar[r]^-{f'} \ar[d]_{r} & A \ar[d]_{q}\\
X \ar[r]^-{f} & A/A[n].\\
} \]
Then $X'$ is a projective scheme equipped with an action of $A$;
moreover, $f'$ is an $A$-equivariant morphism and its fiber at $0$
may be identified to $Y$. It follows that the morphism
\[ A \times Y \longrightarrow X', \quad (a,y) \longmapsto a \cdot y \]
is an isomorphism with inverse 
\[ X' \longrightarrow A \times Y, \quad 
x' \longmapsto (f'(x'), -f'(x') \cdot x'). \] 
So we may identify $X'$ with $A \times Y$; then $r$ is invariant
under the action of $A[n]$ via $g \cdot (a,y) = (a - g, g \cdot y)$.
Since $q$ is an $A[n]$-torsor, so is $r$. In particular, 
$X = (A \times Y)/A[n]$ and the stabilizer in $A$ of any $y \in Y$
is a subgroup scheme of $A[n]$.

By Lemma~\ref{lem:ker}, we may choose a nontrivial $v \in \Aut_{\gp}(A)$
which restricts to the identity on $A[n]$. Then $v \times \id$ is 
an automorphism of $A \times Y$ that commutes with the action of 
$A[n]$. Since $r$ is an $A[n]$-torsor and hence a categorical quotient, 
it follows that $v \times \id \in \Aut(A \times Y)$ 
factors through a unique $u \in \Aut(X)$, which satisfies
$u(a \cdot y) = v(a) \cdot y$ for all $a \in A$ and $y \in Y$.

As $\Aut(X)=A$, we have $u\in A$. For any $a, b \in A$ and $y \in Y$, 
we have $(a + b) \cdot y = b \cdot (a \cdot y)$.  
Choosing $b=u$ in the above formula yields 
$(a + u) \cdot y = u \cdot (a \cdot y) = v(a) \cdot y$. Thus,
 $v(a) - a - u$ fixes every point of $Y$ for any $a \in A$. 
Taking $a = 0$, it follows that $u$ fixes $Y$ pointwise,
and hence $u \in A[n]$. So $v(a) - a \in A[n]$ for any $a \in A$,
i.e., $v - \id$ factors through a homomorphism $A \to A[n]$.

Since $A$ is smooth and connected, 
it follows that $v - \id = 0$, a contradiction.

\section{Proof of Theorem \ref{thm:main}\ref{main2}: first steps}
\label{sec:(ii)}

We assume from now on that $\Aut_{\gp}(A)$ is finite. 
Recall that $q\colon A\to A/A[n]$ is the quotient morphism 
(see Section~\ref{sec:prelim}).

\begin{lemma}\label{lem:inj}$ $

\begin{enumerate}

\item\label{inj1} The map 
$q_* \colon \Aut_{\gp}(A) \to \Aut_{\gp}(A/A[n])$ 
is an isomorphism for any integer $n \geq 1$.

\item\label{inj2} Let $\ell \neq p$ be a prime number. Then
$\rho_{\ell^m} \colon \Aut_{\gp}(A) \to \Aut_{\gp}(A[\ell^m])$
is injective for $m \gg 0$.

\end{enumerate}

\end{lemma}

\begin{proof}
\ref{inj1} Since $\Aut_{\gp}(A/A[n]) \simeq \Aut_{\gp}(A)$ is finite,
it suffices to show that $q_*$ is injective. 
Let $u \in \Aut_{\gp}(A)$ such that $q_*(u) = \id$. Then
$u(a) - a \in A[n]$ for any $a \in A$, that is, $u - \id$ factors
through a homomorphism $A \to A[n]$. As in the very end of 
the proof of Theorem \ref{thm:main}\ref{main1}
the smoothness and connectedness of $A$ yield $u = \id$.

\ref{inj2} 
Let $T_\ell(A) = \lim_{\leftarrow} A[\ell^m]$; then $T_\ell(A)$ is a
$\bZ_\ell$-module and the natural map
$\Aut_{\gp}(A) \to \Aut_{\bZ_\ell}(T_\ell(A))$
is injective (see \cite[\S 19, Thm.~3]{Mumford}). Thus,
$\bigcap_{m \geq 1} \Ker(\rho_{\ell^m}) = \{ \id \}$.
Since the $\Ker(\rho_{\ell^m})$ form a decreasing sequence, we 
get $\Ker(\rho_{\ell^m}) = \{ \id \}$ for $m \gg 0$.
\end{proof}

Next, consider a smooth projective variety $Y$ equipped with an 
action of the finite group $G = A[n]$, for some integer 
$n$ prime to $p$. Then $G$ acts freely on 
$A \times Y$ via $g \cdot (a,y) = (a - g, g \cdot y)$. The quotient 
$X = (A \times Y)/G$ exists and is a smooth projective variety 
(see \cite[\S 7, Thm.]{Mumford}). The $A$-action on $A \times Y$ 
via translation on itself yields an action on $X$. The projection
$\pr_A \colon A \times Y \to A$
yields a morphism
\[ f \colon X \longrightarrow A/G \]
which is $A$-equivariant, where $A$ acts on $A/G$ via the quotient
map $q$. Moreover, $f$ is smooth and its schematic fiber at 
$\bar{0}$ is $G$-equivariantly isomorphic to $Y$.

\begin{lemma}\label{lem:alb}
Assume that $Y$ is rational.

\begin{enumerate}

\item\label{alb1} The map $f$ is the Albanese morphism of $X$.

\item\label{alb2} The neutral component $\Aut^0(Y)$ is a linear 
algebraic group.

\end{enumerate}

\end{lemma}

\begin{proof}
\ref{alb1} Let $B$ be an abelian variety, and $u \colon X \to B$ a morphism.
Composing $u$ with the quotient morphism $A \times Y \to X$
yields a $G$-invariant morphism $v \colon A \times Y \to B$.
As $Y$ is rational, $v$ factors through a morphism $A \to B$, 
which must be $G$-invariant. So $u$ factors through a morphism
$A/G \to B$.

\ref{alb2} By a theorem of Nishi and Matsumura (see \cite{Brion} for a modern
proof), there exist a closed affine subgroup scheme $H \subset \Aut^0(Y)$ 
such that the homogeneous space $\Aut^0(Y)/H$ is an abelian variety, 
and an $\Aut^0(Y)$-equivariant morphism $u \colon Y \to \Aut^0(Y)/H$. 
As $Y$ is rational and $u$ is surjective, this forces $H = \Aut^0(Y)$.
\end{proof}

As a consequence of Lemma~\ref{lem:alb}, if $Y$ is rational then
$f$ induces a homomorphism
\[ f_* \colon \Aut(X) \longrightarrow \Aut(A/G), \]
and hence an exact sequence
\[ 1 \longrightarrow \Aut_{A/G}(X) \longrightarrow \Aut(X)
\stackrel{f_*}{\longrightarrow} A/G \rtimes \Aut_{\gp}(A/G), \]
where $\Aut_{A/G}(X)$ denotes the group of relative automorphisms.
The $A$-action on $X$ yields a homomorphism 
$G \to \Aut_{A/G}(X)$.
Moreover, the image of $f_*$ contains the group $A/G$ of
translations, and hence equals $A/G \rtimes \Gamma$,
where $\Gamma$ denotes the subgroup of $\Aut_{\gp}(A/G)$
consisting of automorphisms which lift to $X$.

\begin{lemma}\label{lem:rel}
Let $G = A[\ell^m]$, where $\ell,m$ satisfy the assumptions of 
Lemma $\ref{lem:inj}\ref{inj2}$.

Let $Y$ be a smooth projective rational $G$-variety such that $\Aut(Y) = G$.

\begin{enumerate}

\item\label{rel1} The map $G \to \Aut_{A/G}(X)$ is an isomorphism.

\item\label{rel2} The group $\Gamma$ is trivial.

\end{enumerate}

\end{lemma}

\begin{proof}
\ref{rel1} Let $u \in \Aut_{A/G}(X)$. Then $u$ restricts to an automorphism
of $Y$ (the fiber of $f$ at $0$), and hence to a unique $g \in G$.
Replacing $u$ with $g^{-1} u$, we may assume that $u$ fixes $Y$
pointwise. For any $a \in A$ and $y \in Y$, we have
$f(u(\overline{(a,y)})) = f(\overline{(a,y)}) = \bar{a}$, 
where $\overline{(a,y)}$ denotes the image of $(a,y)$ in $X$.
As $f$ is $A$-equivariant, it follows that 
$(-a) \cdot u(\overline{(a,y)}) \in Y$. This defines a morphism
\[ F \colon A \times Y \longrightarrow Y, \quad
(a,y) \longmapsto (-a) \cdot u(\overline{(a,y)})  \]
such that $F(0,y) = u(y) = y$ for all $y \in Y$. As $A$ is connected,
this defines in turn a morphism (of varieties) $A \to \Aut^0(Y)$,
which must be constant by Lemma~\ref{lem:alb}\ref{alb2}. So 
$u(\overline{(a,y)}) = a \cdot y = \overline{(a,y)}$
identically, i.e., $u = \id$.

\ref{rel2} Let $\gamma \in \Gamma$; then there exists $u \in \Aut(X)$
such that $f_*(u) = \gamma$. Since $\gamma(\bar{0}) = \bar{0}$,
we see that $u$ stabilizes $Y$; thus, $u \vert_Y = g$ for a unique
$g \in G$. Also, there exists $v \in \Aut_{\gp}(A)$ such that
$q_*(v) = \gamma$ (Lemma~\ref{lem:inj}\ref{inj1}). 
Thus, we have 
$f(u(\overline{a,y})) = \gamma f(\overline{(a,y)}) = \overline{v(a)}$, 
i.e., $(- v(a)) \cdot u(\overline{(a,y)}) \in Y$ for all $a \in A$ 
and $y \in Y$. Arguing as in the proof of \ref{rel1}, it follows that 
\[ u(\overline{(a,y)}) = v(a) \cdot g(y) \]
identically. In particular, $g(a \cdot y) = v(a) \cdot g(y)$ for all $a \in G$
and $y \in Y$. Since $G$ is commutative, we obtain $v(a) = a$
for all $a \in G$. Thus, $v = \id$ by Lemma~\ref{lem:inj}\ref{inj2}.
So $\gamma = \id$ as well.
\end{proof}

\begin{proposition}\label{prop:aut}
Under the assumptions of Lemma~$\ref{lem:rel}$, the $A$-action 
on $X$ yields an isomorphism $A \to \Aut(X)$. If in addition 
$G = \Aut_Y$, then $A \to \Aut_X$ is an isomorphism as well.
\end{proposition}

\begin{proof}
We have a commutative diagram of exact sequences
\[ \xymatrix{
0 \ar[r] & G \ar[r] \ar[d] & 
A \ar[r] \ar[d] & A/G \ar[r] \ar[d] & 0
\\
1 \ar[r] & \Aut_{A/G}(X) \ar[r] &  \Aut(X) \ar[r]^-{f_*} & \Aut(A/G).  \\
} \]
By Lemma~\ref{lem:rel}, the left vertical map is an isomorphism
and the image of $f_*$ is the group $A/G$ of translations. 
This yields the first assertion.

To show the second assertion, it suffices to show that the induced
homomorphism of Lie algebras $\Lie(A) \to \Lie(\Aut_X)$
is an isomorphism when $G = \Aut_Y$. Recall that $\Lie(\Aut_X)$ 
is the space of global sections of the tangent bundle $T_X$ 
(see e.g.~\cite[Lem.~3.4]{MO}). Moreover, 
as $f$ is smooth, we have an exact sequence
\[ 0 \longrightarrow T_f \longrightarrow T_X 
\stackrel{df}{\longrightarrow} f^*(T_{A/G}) \longrightarrow 0, \]
where $T_f$ denotes the relative tangent bundle.
Since $T_{A/G}$ is the trivial bundle with fiber $\Lie(A/G)$,
this yields an exact sequence
\[ 0 \longrightarrow H^0(X,T_f) \longrightarrow H^0(X,T_X)
\longrightarrow \Lie(A/G) \]
such that the composition 
$\Lie(A) \to H^0(X,T_X) \to \Lie(A/G)$ is $\Lie(q)$.
So it suffices in turn to show that $H^0(X,T_f) = 0$.

We have a cartesian diagram 
\[ \xymatrix{
A \times Y \ar[r]^-{\pr_A} \ar[d] & A \ar[d]_{}\\
X \ar[r]^-{f} & A/G,\\
} \]
where the vertical arrows are $G$-torsors.
This yields an isomorphism
\[ H^0(X,T_f) \simeq H^0(A \times Y, T_{\pr_A})^G \]
and hence
\[ H^0(X,T_f) \simeq H^0(A \times Y, \pr_Y^*(T_Y))^G
\simeq (\cO_A(A) \otimes H^0(Y,T_Y))^G
\simeq H^0(Y,T_Y)^G. \]
As $G = \Aut_Y$, we have $H^0(Y,T_Y) = \Lie(G) = 0$; this
completes the proof. 
\end{proof}

\section{Proof of Theorem \ref{thm:main}\ref{main2}: the construction of $Y$}
\label{Sec:Y}
In this section, we fix integers $n,r\ge 2$, where $p$ 
does not divide $n$, and construct a smooth projective rational variety $Y$ 
of dimension $r$ such that $\Aut_Y=(\mathbb{Z}/n)^r$.

We define 
\[G=\{(\mu_1,\ldots,\mu_r)\in \k^r\mid \mu_i^n=1
\text{ for each } i\in \{1,\ldots,r\}\}\simeq (\mathbb{Z}/n)^r\]
 and let $G$ act on $(\p^1)^r$ by 
 \[\begin{array}{ccc}
 G \times (\p^1)^r&\to & (\p^1)^r\\
( (\mu_1,\ldots,\mu_r),([ u_1:v_1],\ldots,[ u_r:v_r]))&
\mapsto & ([ u_1:\mu_1v_1],\ldots,[ u_r:\mu_rv_r])\end{array}\]

For each $i\in \{1,\ldots, r\}$, we denote by $\ell_i\subset (\p^1)^r$ 
the closed curve isomorphic to $\p^1$ given by the image of
 \[\begin{array}{ccc}
 \p^1&\to & (\p^1)^r\\
([u:v])&\mapsto & ([0:1],\ldots,[0:1],[u:v],[0:1],\ldots,[0:1])\end{array}\]
where the $[u:v]$ is at the place $i$. The curves 
$\ell_1,\ldots,\ell_r\subset (\p^1)^r$ generate the cone of curves of $(\p^1)^r$.

For each $i\in \{1,\ldots, r\}$, the curve $\ell_i$ is stable by $G$ 
and the action of $G$ on $\ell_i$ corresponds to a cyclic action of order 
$n$ on $\p^1$, given by $[u:v]\mapsto [\mu u:v]$, where $\mu\in \k$, $\mu^n=1$. 
All orbits are of size $n$, except the two fixed points $[0:1]$ and $[1:0]$.

We choose $s=(s_1,\ldots,s_r)$ to be a sequence of positive integers, all 
distinct, such that $s_i\cdot n\ge 3$ for each $i$ if $r=2$, and consider a 
finite subset 
\[ \Delta\subset \ell_1\cup \cdots\cup \ell_r\subset (\p^1)^r, \] 
stable
by $G$, given by a union of orbits of size $n$. 
For each $i\in \{1,\ldots,r\}$, we define $\Delta_i\subset \ell_i$ 
to be a union of exactly $s_i\ge 1$ orbits of size $n$, and choose then 
$\Delta=\bigcup_{i=1}^r \Delta_i$. We moreover choose the points such that 
the group $H=\{h\in \Aut(\p^1)\mid h(\Delta_i)=\Delta_i, h([0:1])=[0:1]\}$ 
only consists of $\{[u:v]\mapsto [\mu u:v]\mid u^n=1\}$.  
As the unique point of intersection of any two distinct
$\ell_i$ is fixed by $G$,  each point of $\Delta$ lies on exactly one 
of the curves $\ell_i$. This gives 
\[\Delta=\uplus_{i=1}^r \Delta_i\]

Let $\pi\colon Y\to (\p^1)^r$ be the blow-up of $\Delta$. 
As $\Delta$ is $G$-invariant, the action of $G$ lifts to 
an action on $Y$. We want to prove that the resulting homomorphism
$G \to \Aut_Y$ is an isomorphism.

\subsection{Intersection on $(\p^1)^r$}
For $i=1,\ldots,r$, we denote by $H_i\subset (\p^1)^r$ the hypersurface 
given by
\[H_i=\{([u_1:v_1],\ldots,[u_r:v_r])\in (\p^1)^r\mid u_i=0\}.\]
Then $H_1,\ldots,H_r$ generate the cone of effective divisors on $(\p^1)^r$, 
and we have
\[H_i\cdot \ell_i=1, H_i\cdot \ell_j=0\]
for all $i,j \in \{1,\ldots,r\}$ with $i\not= j$. Moreover, 
the canonical divisor class of $(\p^1)^r$ satisfies
$K_{(\p^1)^r}=-2H_1-2H_2-\cdots -2H_r$, so $K_{(\p^1)^r}\cdot \ell_i=-2$ 
for each $i \in \{1,\ldots,r\}$.

We also observe that $\ell_i\subset H_j$ for all $i,j \in \{1,\ldots,r\}$ 
with $i\not= j$ and that $\ell_i\not\subset H_i$.

\subsection{Intersection on $Y$}

For $i=1,\ldots,r$, denote by $\tilde{\ell}_i,\tilde{H}_i\subset Y$ 
the strict transforms of $\ell_i$ and $H_i$. 

For each $p\in \Delta$, we denote by $E_p=\pi^{-1}(p)$ the exceptional 
divisor, isomorphic to $\p^{r-1}$, and choose a line $e_p\subset E_p$.

A basis of the Picard group of $Y$ is given by  the union of 
$\tilde{H}_1,\ldots,\tilde{H}_r$ and of all exceptional divisors $E_p$, 
with $p\in \Delta$. A basis of the vector space of curves 
(up to numerical equivalence) is given by  
$\tilde{\ell}_1,\ldots,\tilde{\ell}_r$ and by all $e_p$ with $p\in \Delta$. 
We have
\begin{equation*}e_p\cdot E_p=-1, e_p\cdot E_q=0\end{equation*}
for all $p,q\in \Delta$, $p\not=q$.

\begin{lemma}\label{LemmaHiellj}
For all $i,j\in \{1,\ldots,r\}$ with $i\not=j$, the following hold:
\begin{enumerate}
\item\label{tildeHi}
$\tilde{H}_i=\pi^*(H_i)-\sum\limits_{p\in \Delta\cap H_i} E_p
=\pi^*(H_i)-\sum\limits_{s\not=i}\sum\limits_{p\in \Delta_s} E_p.$
\item\label{tildeelliDelta}
$\tilde{\ell}_i\cdot E_p=1$ if $p\in \Delta_i$ and 
$\tilde{\ell}_i\cdot E_p=0$ if $p\in \Delta\setminus \Delta_i$.
\item\label{tildeHielli}
$\tilde{H}_i\cdot \tilde{\ell}_i=1$.
\item\label{tildeHiellj}
$\tilde{H}_i\cdot \tilde{\ell}_j=-\lvert \Delta_j\rvert=-n s_j$.
\end{enumerate}
\end{lemma}

\begin{proof}
\ref{tildeHi} follows from the fact that $H_i$ is a smooth hypersurface 
of $(\p^1)^r$ and that $\Delta\cap H_i=\bigcup\limits_{s\not=i} \Delta_s$.

\ref{tildeelliDelta}: follows from the fact that $\ell_i$ is a smooth curve, 
passing through all points of $\Delta_i$ and not through any point 
of $\Delta\setminus \Delta_i$.

\ref{tildeHielli}: 
With \ref{tildeHi} and \ref{tildeelliDelta}, we get  
$\tilde{H}_i\cdot \tilde{\ell}_i= H_i\cdot \ell_i =1$.

\ref{tildeHiellj}:
With \ref{tildeHi} and \ref{tildeelliDelta}, we get  
$\tilde{H}_i\cdot \tilde{\ell}_j=H_i\cdot \ell_j-\lvert \Delta_j\rvert
=-\lvert \Delta_j\rvert=-ns_j$.
\end{proof}

\begin{lemma}\label{Lemm:gammapj}
For all $i\in \{1,\ldots, r\}$ and each $p\in \Delta\setminus\Delta_i$, 
we take the irreducible curve $\gamma_{p,i}\subset (\p^1)^r$ passing through 
$p$ and being numerically equivalent to $\ell_i$. 

\begin{enumerate}
\item\label{gammapj}
Let $j\in \{1,\ldots,r\}$ be such that $p\in \Delta_j$.
The $j$-th coordinate of $\gamma_{p,i}$ is the one of $p$, 
its $i$-th coordinate is free, and all others are $[0:1]$. 
\item\label{gammapjeq}
The strict transform $\tilde{\gamma}_{p,i}$ of $\gamma_{p,i}$ on $Y$ 
is numerically equivalent to 
$\tilde{\ell}_i+\sum\limits_{q\in \Delta_i} e_q-e_p$ and satisfies 
$\tilde{\gamma}_{p,i}\cdot E_p=1$ 
and $\tilde{\gamma}_{p,i}\cdot E_q=0$ for all $q\in \Delta\setminus \{p\}$.
\end{enumerate}
\end{lemma}
\begin{proof}\ref{gammapj}:
We write $p=(p_1,\ldots,p_r)\in (\p^1)^r$. Since 
$\gamma_{p,i}\subset (\p^1)^r$ is a curve equivalent to $\ell_i$ 
and passing through $p$, it has to be 
\[\gamma_{p,i}
=\{(p_1,\ldots,p_{i-1},t,p_{i+1},\ldots,p_r)\in (\p^1)^r\mid t\in \p^1\}
\simeq \p^1.\]
Moreover, for each $s\in \{1,\ldots,r\}\setminus \{j\}$, we have 
$p_s=[0:1]$, as $p\in \Delta_j\subset \ell_j$. This completes 
the proof of~\ref{gammapj}.

\ref{gammapjeq}: We want to prove that 
$\tilde{\gamma}_{p,i}\equiv \tilde{\ell}_i+\sum\limits_{q\in \Delta_i} e_q-e_p$. 
For each divisor $D$ on $(\p^1)^r$, we have 
\begin{align*}\tilde{\gamma}_{p,i}\cdot \pi^*(D)&
=\pi(\tilde{\gamma}_{p,i})\cdot D=\gamma_{p,i}\cdot D\\
(\tilde{\ell}_i-e_p)\cdot \pi^*(D)&=\pi(\tilde{\ell}_{i})\cdot D
=\ell_i\cdot D=\gamma_{p,i}\cdot D\end{align*}
We moreover have (with Lemma~\ref{LemmaHiellj}\ref{tildeelliDelta})
\begin{align*}\tilde{\gamma}_{p,i}\cdot E_p&=1
=E_p\cdot (\tilde{\ell}_i+\sum\limits_{q\in \Delta_i} e_q-e_p),\\
\tilde{\gamma}_{p,i}\cdot E_{p'}&=0
=E_{p'}\cdot (\tilde{\ell}_i+\sum\limits_{q\in \Delta_i} e_q-e_p), 
\text{ for all }p'\in \Delta\setminus \{p\}.\end{align*}
\end{proof}

\begin{lemma}\label{Lemm:ConeGen}
Let $\gamma\subset Y$ be an irreducible curve. Then, one of the following holds:
\begin{enumerate}
\item\label{ConeGen1}
We have $\gamma\equiv de_p$ for some $d\ge 1$ and some $p\in \Delta$
$($where $\equiv$ denotes numerical equivalence$)$;
\item\label{ConeGen2}
There are non-negative integers $a_1,\ldots,a_r$ and $\{\mu_p\}_{p\in \Delta}$ 
such that
\[\gamma\equiv  \sum_{i=1}^r a_i \tilde{\ell}_i+ \sum_{p\in \Delta} \mu_p e_p\]
and such that $a_1+\cdots +a_r\ge 1$.
 \item\label{ConeGen3}
There are $j\in \{1,\ldots,r\}$, $q\in \Delta_j$ and integers 
$a_1,\ldots,a_r\ge 0$  such that 
\[\gamma\equiv a_j e_q+\sum\limits_{i\not=j} a_i \tilde{\gamma}_{q,i}\]
and such that $\sum_{i\not=j} a_i\ge 1$.
\end{enumerate}
\end{lemma}

\begin{proof}
Suppose first that $\gamma$ is contained in some $E_p$, where $p\in \Delta$. 
In this case, $\gamma$ is a curve of degree $d\ge 1$ in the projective space 
$E_p\simeq \p^{r-1}$ (if $r=2$, then $\gamma=e_p=E_p$  and $d=1$), and thus 
$\gamma\equiv d e_p$. This gives Case~\ref{ConeGen1}.

We may now assume that $\gamma$ is not contained in $E_p$ for any $p\in \Delta$. 
Hence, $\gamma$ is the strict transform of the irreducible curve 
$\pi(\gamma)\subset (\p^1)^r$, numerically equivalent to 
$\sum_{i=1}^r a_i \ell_i$, with $a_1,\ldots,a_r\ge 0$ and 
$\sum_{i=1}^r a_i\ge 1$. For each $p\in \Delta$, we write 
$\epsilon_p= E_p\cdot \gamma \ge 0$. 

We first prove that
\begin{equation}\label{gammaequiv}\tag{$\spadesuit$}
\gamma\equiv \sum_{i=1}^r a_i \tilde{\ell}_i
+\sum_{i=1}^r\sum_{p\in \Delta_i} (a_i-\epsilon_p) e_p.
\end{equation}
Intersecting both sides of \eqref{gammaequiv} with the divisor $\pi^*(D)$, 
for any divisor $D$ on $(\p^1)^r$, gives 
$\pi(\gamma)\cdot D=\sum a_i\ell_i \cdot D$. Moreover, for each $p\in \Delta$, 
there is $j\in\{1,\ldots,r\}$ such that $p\in \Delta_j$. 
Intersecting $E_p$ with both sides of \eqref{gammaequiv}, we obtain 
$E_p\cdot \gamma=\epsilon_p
\stackrel{\text{Lemma~\ref{LemmaHiellj}\ref{tildeelliDelta}}}{=}
E_p\cdot (\sum\limits_{i=1}^r a_i \tilde{\ell}_i
+\sum\limits_{i=1}^r\sum\limits_{p\in \Delta_i} (a_i-\epsilon_p) e_p)$. 
This completes the proof of \eqref{gammaequiv}.

For each $p\in \Delta$, we denote by $i\in \{1,\ldots, r\}$ the integer 
such that $p\in \Delta_i$ and by $H_{p}\subset (\p^1)^r$ the hypersurface 
consisting of points $q\in (\p^1)^r$ having the same $i$-th coordinate as $p$. 
Hence $p_i\in H_{p}$,  $H_{p}\cap \Delta=\{p\}$ and $H_{p}\sim H_i$. 
The strict transform of $H_{p}$, that we write $\tilde{H}_{p}$, satisfies 
$\tilde{H}_{p}\sim \pi^*(H_i)-E_p$. This gives
\begin{equation} \tag{$\heartsuit$}
\tilde{H}_{p}\cdot \gamma=a_i-E_p\cdot \gamma =a_i-\epsilon_p .\label{gammaHp}
\end{equation}

Suppose first that $ \tilde{H}_{p}\cdot \gamma\ge 0$ for each $p\in \Delta$. 
This means (with \eqref{gammaHp}), that $a_i-\epsilon_p\ge 0$ for each 
$i\in \{1,\ldots,r\}$ and each $p\in \Delta_i$. Hence all coefficients in 
\eqref{gammaequiv} are non-negative, so we obtain ~\ref{ConeGen2}. 

Suppose now that $ \tilde{H}_{q}\cdot \gamma< 0$ for some $q\in \Delta$. 
This implies that $\gamma\subset \tilde{H}_q$. As $H_{q}\cap \Delta=\{q\}$, 
we obtain $E_p\cap \tilde{H}_q=\emptyset$ for each $p\in \Delta\setminus \{q\}$, 
which yields  $\epsilon_p=E_p\cdot \gamma=0$. Writing $j\in\{1,\ldots,r\}$ 
the element such that $q\in \Delta_j$, the $j$-th component of 
$\pi(\gamma)\subset (\p^1)^r$ is constant, so 
$a_j=\pi^*(H_j)\cdot \gamma=H_j\cdot \pi(\gamma)=0$. We now prove that
\begin{equation}\tag{$\diamondsuit$}
\gamma\equiv (-\epsilon_q+\sum\limits_{i\not=j} a_i ) e_q
+\sum\limits_{i\not=j} a_i \tilde{\gamma}_{q,i}\label{sumgammaqi}
\end{equation}
Intersecting both sides of \eqref{sumgammaqi} with the divisor $\pi^*(D)$, 
for any divisor $D$ on $(\p^1)^r$, gives 
$\pi(\gamma)\cdot D=\sum a_i\ell_i \cdot D$. Intersecting $E_q$ with both sides 
gives $\epsilon_q=\epsilon_q$, since $E_q\cdot \tilde{\gamma}_{q,i}=1$ 
for each $i\not=j$ (Lemma~\ref{Lemm:gammapj}\ref{gammapjeq}). 
Intersecting with $E_p$ for $p\in \Delta\setminus \{q\}$ gives 
$\epsilon_p=0$. This completes the proof of \eqref{sumgammaqi}.

As the $j$-th component of $\pi(\gamma)\subset (\p^1)^r$ is constant, 
there is an integer $i\in \{1,\ldots,r\}\setminus \{j\}$ such that 
the $i$-th component of $\pi(\gamma)$ is not constant. This implies 
that $\pi(\gamma)\not\subset H_i$, so $\tilde{\gamma}\not\subset \tilde{H}_i$. 
We obtain
\[0\le \tilde{H}_i\cdot \gamma
\stackrel{\text{Lemma~\ref{LemmaHiellj}\ref{tildeHi}}}
=(\pi^*(H_i)-\sum\limits_{s\not=i}\sum\limits_{p\in \Delta_s} E_p)\cdot \gamma
=a_i-\epsilon_q.\]
Hence, the coefficents of \eqref{sumgammaqi} are non-negative, giving \ref{ConeGen3}.
%
%
%
%
%
%
%
%
%
%
%
\end{proof}

\begin{proposition}\label{Prop:Equiv}
Let $\gamma\subset Y$ be an irreducible curve. Then, the following are equivalent:
\begin{enumerate}
\item\label{equiv1}
 For all effective $1$-cycles $\gamma_1,\gamma_2$ on $Y$ such that  
 $\gamma\equiv \gamma_1+\gamma_2$, we have $\gamma_1=0$ or $\gamma_2=0$.
 \item\label{equiv2}
 $\gamma$ is numerically equivalent to $\tilde{\ell}_i$ for some 
 $i\in \{1,\ldots,r\}$, to  $\tilde{\gamma}_{p,i}$ for some 
 $i\in \{1,\ldots,r\}, p\in \Delta\setminus \Delta_i$, or to $e_p$ 
 for some $p\in \Delta$.
 \item\label{equiv3}
 $\gamma$ is either equal to $\tilde{\ell}_i$ for some $i\in \{1,\ldots,r\}$, 
 or equal to  
 $\tilde{\gamma}_{p,i}$ for some $i\in \{1,\ldots,r\}, p\in \Delta\setminus \Delta_i$, 
 or is a line in $E_p$, for some $p\in \Delta$.
 \end{enumerate}
\end{proposition}

\begin{proof}

$\ref{equiv1}\Rightarrow \ref{equiv2}$: 
By Lemma~\ref{Lemm:ConeGen}, $\gamma\equiv\gamma_1+\cdots+\gamma_s$ 
where $s\ge 1$ and where $\gamma_1,\ldots,\gamma_s$ belong to 
$\{\tilde{\ell}_i\mid i\in \{1,\ldots,r\} \}\cup 
\{e_p\mid p\in \Delta\}\cup \{\tilde{\gamma}_{p,i}
\mid i\in \{1,\ldots,r\},p\in \Delta\setminus \Delta_i\}$. 
As \ref{equiv1} is satisfied, we have $s=1$, which implies \ref{equiv2}.

$\ref{equiv2}\Rightarrow \ref{equiv3}$: 
Suppose first that $\gamma\equiv e_p$ for some $p\in \Delta$. 
For an ample divisor $D$ on $(\p^1)^r$, we have 
$0=e_p\cdot \pi^*(D)=\pi_*(\gamma)\cdot D$, which implies that 
$\gamma$ is contracted by $\pi$. Hence, $\gamma$ is a curve 
of degree $d\ge 1$ in some $E_q$, $q\in \Delta$, and is thus 
equivalent to $d e_q$. As $-1=E_p\cdot e_p=E_p\cdot \gamma$, 
we have $q=p$ and $d=1$.

Suppose now that $\gamma\equiv \tilde{\ell}_i$ for some $i\in \{1,\ldots,r\}$. 
For each $j\in \{1,\ldots,r\}$ with $j\not=i$, we have 
$\tilde{H}_i\cdot \gamma=\tilde{H}_i\cdot \tilde{\ell}_j
\stackrel{\text{Lemma~\ref{LemmaHiellj}\ref{tildeHiellj}}}{=}-n s_j<0$. 
Hence, $\pi(\gamma)\subset \bigcap_{j\not=i } H_j=\ell_i$. 
As $\pi(\gamma)\cdot H_i=\pi^*(H_i)\cdot \gamma=\pi^*(H_i)\cdot \tilde{\ell}_i=1$, 
we have $\pi(\gamma)=\ell_i$ and $\tilde{\gamma}=\tilde{\ell}_i$.

In the remaining case, $\gamma\equiv\tilde{\gamma}_{p,i}$ for some 
$i\in \{1,\ldots,r\}$ and some $p\in \Delta\setminus \Delta_i$. 
Hence, $\pi(\gamma)$ is numerically equivalent to $\pi(\tilde{\gamma}_{p,i})$, 
which is equivalent to $\ell_i$ (Lemma~\ref{Lemm:gammapj}\ref{gammapjeq}). 
Hence, all coordinates of $\pi(\gamma)$ except the $i$-th one are constant. 
As $\gamma\cdot E_p=\tilde{\gamma}_{p,i}\cdot E_p=1$ 
(again by Lemma~\ref{Lemm:gammapj}\ref{gammapjeq}), 
the point $p$ belongs to both $\pi(\gamma)$ and ${\gamma}_{p,i}$, 
which yields $\pi(\gamma)={\gamma}_{p,i}$ and thus $\gamma=\tilde{\gamma}_{p,i}$.

$\ref{equiv3}\Rightarrow \ref{equiv1}$: We take effective $1$-cycles 
$\gamma_1,\gamma_2$ on $Y$ such that  $\gamma\equiv \gamma_1+\gamma_2$ 
and prove that one of the two is zero, using $\ref{equiv3}$. 

For each $i\in \{1,\ldots,r\}$, we write $a_i=\pi^*(H_i)\cdot \gamma$, 
$b_i=\pi^*(H_i)\cdot \gamma_1$ and $c_i=\pi^*(H_i)\cdot \gamma_2$ 
and obtain $a_i=b_i+c_i$.  As $H_i$ is nef, $\pi^*(H_i)$ is nef, so 
$a_i,b_i,c_i\ge 0$. Moreover, $\gamma$ satisfying $\ref{equiv3}$, 
we have $\sum_{i=1}^r a_i=1$, which implies that, up to exchanging 
$\gamma_1$ and $\gamma_2$, we may assume that 
$\sum_{i=1}^r a_i=\sum_{i=1}^r b_i$ and 
$c_i=0$ for $i = 1, \ldots,r$. 
In particular, $\gamma_2$ is a sum of irreducible curves contained in 
the exceptional divisors $E_p$, $p\in \Delta$. 

Suppose first that $\gamma=e_q$ for some $q\in \Delta$. This gives 
$\sum_{i=1}^r a_i=\sum_{i=1}^r b_i = 0$, 
which implies that both $\gamma_1$ and $\gamma_2$ are  
sums of irreducible curves contained in 
the exceptional divisors $E_p$, $p\in \Delta$. 
For each $p'\in \Delta$ and each irreducible curve $c\subset E_{p'}$ 
of degree $d\ge 1$ we get $\sum_{p\in \Delta} E_p\cdot c=-d$. As 
$\sum_{p\in \Delta} E_p\cdot \gamma=-1$, this gives $\gamma_1=0$ or $\gamma_2=0$.

We may now take $s\in \{1,\ldots,r\}$ and either $\gamma=\tilde{\ell}_s$ or  
$\gamma=\tilde{\gamma}_{p,s}$ for some $p\in \Delta\setminus \Delta_s$. 
This gives $b_s=1$ and $b_i=0$ for all 
$i\in \{1,\ldots,r\}\setminus \{s\}$. 
Lemma~\ref{Lemm:ConeGen} implies that $\gamma_1$ is
equivalent to a sum of curves contained in 
$\{\tilde{\ell}_i\mid i\in \{1,\ldots,r\} \}\cup 
\{e_p\mid p\in \Delta\}\cup \{\tilde{\gamma}_{p,i}
\mid i\in \{1,\ldots,r\},p\in \Delta\setminus \Delta_i\}$. 
As $b_s=1$ and $b_i=0$ for all $i\in \{1,\ldots,r\}\setminus \{s\}$, 
we have $\gamma_1\equiv \alpha+\beta$, where $\alpha$ is either equal to 
$\tilde{\ell}_s$ or $\tilde{\gamma}_{p,s}$ for some 
$p\in \Delta\setminus \Delta_s$ and where $\beta$ is a non-negative sum of 
$e_p, p\in \Delta$. For each $p\in \Delta$, we obtain 
\[E_p\cdot \gamma =
E_p\cdot \alpha +E_p\cdot \beta+E_p\cdot \gamma_2\le E_p\cdot \alpha.\]
We now use the fact that we know the  intersection of $\alpha$ and $\gamma$ 
with $E_p$ (which is given either by 
Lemma~\ref{LemmaHiellj}\ref{tildeelliDelta}
or by Lemma~\ref{Lemm:gammapj}\ref{gammapjeq}, depending if the curve is equal 
to $\tilde{\ell}_s$ or $\tilde{\gamma}_{p,s}$).

If $\gamma=\tilde{\gamma}_{p,s}$ for some $p\in \Delta\setminus \Delta_s$, 
then $1=E_p\cdot \gamma\le E_p\cdot \alpha$, which implies that 
$\alpha=\tilde{\gamma}_{p,s}$. If $\gamma=\tilde{\gamma}_s$, then 
$1=E_q\cdot \gamma\le E_q\cdot \alpha$ for each $q\in \Delta_s$, 
which implies that $\alpha=\tilde{\gamma}_s$. In both cases, we get 
$\alpha=\gamma$, which implies that $E_p\cdot \gamma_2=0$ for each 
$p\in \Delta$, and thus that $\gamma_2=0$, as desired.
\end{proof}

\begin{theorem}
The map $G \to \Aut_Y$ is an isomorphism.
\end{theorem}

\begin{proof}
We first show that $G \stackrel{\sim}{\to} \Aut(Y)$. 
Let $\alpha\in \Aut(Y)$. For each irreducible curve $\gamma\subset Y$ 
that satisfies Proposition~\ref{Prop:Equiv}\ref{equiv1}, the curve 
$\alpha(\gamma)$ also satisfies Proposition~\ref{Prop:Equiv}\ref{equiv1}. 
Hence, the union $F\subset Y$ of all curves satisfying this assertion 
is also stable by $\Aut(Y)$.

By Proposition~\ref{Prop:Equiv}, we have
\[F=(\bigcup_{p\in \Delta } E_p) \cup (\bigcup_{i=1}^r \tilde{\ell}_i) \cup 
(\bigcup_{i=1}^r (\bigcup_{p\in \Delta\setminus \Delta_i}\tilde{\gamma}_{p,i})).\] 
We observe that the above union is the decomposition of $F$ 
into irreducible components. Hence, $\alpha$ permutes the irreducible components.  
We now make the following observations:
\begin{enumerate}
\item
For each $i\in \{1,\ldots,r\}$, $\tilde{\ell}_i$ intersects exactly 
$n\cdot s_i$ other irreducible components of $F$, namely the $E_p$ with 
$p\in \Delta$.
\item
For each $p\in \Delta_i$, the divisor $E_p$ intersects exactly $r$ other 
irreducible components of $F$, namely the curve $\tilde{\ell}_i$  and 
the curves $\tilde{\gamma}_{p,j}$ with $j\in \{1,\ldots,r\}\setminus \{i\}$.
\item
For each 
$i\in \{1,\ldots,r\}$ and $p \in \Delta \setminus \Delta_i$, the curve
$\tilde{\gamma}_{p,i}$ intersects exactly $n\cdot s_i+1$ other  
irreducible components of $F$. Writing $j\in \{1,\ldots,r\}$ 
the element such that $p\in \Delta_j$, the curve intersects $E_p$ 
and all curves $\tilde{\gamma}_{q,j}$ for each $q\in \Delta_i$.
\end{enumerate}

If $r\ge 3$, the exceptional divisors $E_p$ are the irreducible components 
of maximal dimension of $F$, so $g$ permutes them. If $r=2$, then 
$g$ also permutes the $E_p$, as these are the only irreducible components 
of $F$ that intersect exactly $2$ other irreducible components of $F$ 
(we assumed $n\cdot s_i\ge 3$ for each $i$ in the case $r=2$). 
In any case, $g$ permutes the exceptional divisors $E_p$ and is thus 
the lift of an automorphism $\hat{g}$ of $(\p^1)^r$: we observe that 
the birational self-map $\hat{g}=\pi g \pi^{-1}$ of $(\p^1)^r$
restricts to an automorphism on the complement of $\Delta$, and as 
$\Delta$ has codimension $\ge 2$, $\hat{g}$ is an automorphism. 
We then use again the three observations above to see that 
$g(\tilde{\ell}_i)=\tilde{\ell}_i$ for each $i\in \{1,\ldots,r\}$, 
as the $s_i$ are all distinct. Hence, 
$\hat{g}(\ell_i)=\hat{g}(\ell_i)$ for each $i$. 
This implies that $\hat{g}$ is of the form
 \[\begin{array}{ccc}
(\p^1)^r&\to & (\p^1)^r\\
( (\mu_1,\ldots,\mu_r),)&\mapsto & 
([ u_1:\mu_1 v_1+\kappa_1u_1],\ldots,[ u_r:\mu_r v_r+\kappa_ru_r])
\end{array}\]
for some $\mu_1,\ldots,\mu_r\in \k^*$ and $\kappa_1,\ldots,\kappa_r\in \k$.

For each $i\in \{1,\ldots,r\}$, the restriction of $\hat{g}$ to $\ell_i$ 
corresponds to the automorphism $[u:v]\mapsto [ u_i:\mu_i v_1+\kappa_iu_i]$. 
As it has to stabilize the set $\Delta_i$, we have $\kappa_i=0$ and 
$\mu_i\in\k^*$ is of order $n$. 
This yields the isomorphism $G \simeq \Aut(Y)$. 

To complete the proof, it suffices to show that $\Aut_Y$ is constant,
or equivalently that its Lie algebra is trivial. 
(We refer to \cite[\S 2.1]{Martin} for background on infinitesimal 
automorphisms and vector fields).
Recall that $\Lie(\Aut_Y) = H^0(Y,\cT_Y)$, where $\cT_Y$ denotes the
tangent sheaf. In other terms, $\Lie(\Aut_Y)$ consists of the 
global vector fields on $Y$. Denoting by 
$E = \uplus_{p \in \Delta} E_p$ the exceptional divisor, we have 
an exact sequence of sheaves on $Y$
\[ 0 \longrightarrow \cT_{Y,E} \longrightarrow \cT_Y 
\longrightarrow \bigoplus_{p \in \Delta} 
(i_{E_p})_*(\cN_{E_p/Y}) \longrightarrow 0, \]
where $\cT_{Y,E}$ is the sheaf of vector fields that are tangent to $E$,
and $\cN_{E_p/Y}$ denotes the normal sheaf. 
Moreover, for any $p \in \Delta$, we have 
$E_p \simeq \bP^{r-1}$ and this identifies $\cN_{E_p/Y}$ with 
$\cO_{\bP^{r-1}}(-1)$; thus, $H^0(E_p,\cN_{E_p/Y}) = 0$. 
As a consequence, 
$H^0(Y,\cT_{Y,E}) \stackrel{\sim}{\to} H^0(Y,\cT_Y)$. 
Viewing vector fields as derivations of the structure sheaf $\cO_Y$,
this yields 
\[ \Der(\cO_Y,\cO_Y(-E)) \stackrel{\sim}{\to} \Der(\cO_Y), \]
where the left-hand side denotes the Lie algebra of derivations
which stabilize the ideal sheaf of $E$.

The blow-up $\pi : Y \to (\bP^1)^r$ contracts $E$ to $\Delta$ and 
satisfies $\pi_*(\cO_Y) = \cO_{(\bP^1)^r}$; also, 
$\pi_*(\cO_Y(-E)) = \cI_{\Delta}$ (the ideal sheaf of $\Delta$).
So $\pi$ induces a homomorphism of Lie algebras
$\pi_* : \Der(\cO_Y) \to \Der(\cO_{(\bP^1)^r})$,
which is injective as $\pi$ is birational. Moreover, $\pi_*$
sends $\Der(\cO_Y,\cO_Y(-E))$ into
$\Der(\cO_{(\bP^1)^r}, \cI_{\Delta})$, the Lie algebra of
vector fields on $(\bP^1)^r$ which vanish at each $p \in \Delta$.
So it suffices to show that each such vector field is zero.

We have
\[ \Der(\cO_{(\bP^1)^r}) = H^0((\bP^1)^r, \cT_{(\bP^1)^r}) 
 = \bigoplus_{i=1}^r H^0(\bP^1,\cT_{\bP^1}) = \Lie(\Aut_{\bP^1})^r. 
\]
Moreover, $\Lie(\Aut_{\bP^1}) = M_2(\k)/\k \, \id$, the quotient of
the Lie algebra of $2 \times 2$ matrices by the scalar matrices. 
Let $\xi = (\xi_1,\ldots,\xi_r) \in \Der(\cO_{(\bP^1)^r})$,
with representative $(A_1,\ldots,A_r) \in M_2(\k)^r$. Then $\xi$
vanishes at $p = ([x_1:y_1],\ldots,[x_r:y_r])$ if and only if 
$(x_i,y_i)$ is an eigenvector of $A_i$ for each $i \in \{ 1,\ldots,r\}$. 
Thus, if $\xi \in \Der(\cO_{(\bP^1)^r},\cI_{\Delta})$, then
$(0,1)$ is an eigenvector of each $A_i$, i.e., $A_i$ is lower
triangular. In addition, each point of $\Delta_i$ yields an
eigenvector of $A_i$. So each $A_i$ is scalar, and $\xi = 0$
as desired.
\end{proof}

\bibliographystyle{amsalpha}

\end{document}